\documentclass[12pt, reqno]{article}
\usepackage{amsmath}
\usepackage{amssymb}
\usepackage{latexsym,bm}
\usepackage{amsthm}
\usepackage{graphicx}

\usepackage[top=2.4cm,bottom=2cm,left=2.4cm,right=2.4cm]{geometry}
\usepackage{algorithm}
\usepackage{algorithmic}
\usepackage{cite}
\usepackage[colorlinks,urlcolor=blue]{hyperref}
\allowdisplaybreaks
\usepackage{dsfont}
\usepackage{multirow}
\usepackage{threeparttable}
\usepackage{slashbox}

\newtheorem{thm}{Theorem}[section]

\newtheorem{lem}{Lemma}[section]

\theoremstyle{definition}
\newtheorem{definition}{Definition}[section]
\theoremstyle{remark}
\newtheorem{rmk}{Remark}[section]
\numberwithin{equation}{section}
\theoremstyle{example}

\numberwithin{equation}{section} \numberwithin{algorithm}{section}

\begin{document}

\begin{center}

\textbf{\large A first-order splitting method for solving a large-scale composite convex optimization problem }

\end{center}

\begin{center}
Yu-Chao Tang$^{1,2}$, Guo-Rong Wu$^{3}$, Chuan-Xi Zhu$^{1,2}$
\end{center}

\begin{center}
1. School of Management, Nanchang University, Nanchang 330031, P.R. China\\
2. Department of Mathematics, Nanchang University, Nanchang 330031,
P.R. China\\
3. Department of Radiology and BRIC, University of North Carolina at Chapel Hill, NC 27599, USA

\end{center}

\bigskip

\begin{abstract}
The forward-backward operator splitting algorithm is one of the most important methods for solving the optimization problem of the sum of two convex functions, where one is differentiable with a Lipschitz continuous gradient and the other is possibly nonsmooth but proximable. It is convenient to solve some optimization problems in the form of dual or primal-dual problems. Both methods are mature in theory. In this paper, we construct several efficient first-order splitting algorithms for solving a multi-block composite convex optimization problem. The objective function includes a smooth function with a Lipschitz continuous gradient, a proximable convex function that may be nonsmooth, and a finite sum of a composition of a proximable function and a bounded linear operator. To solve such an optimization problem, we transform it into the sum of three convex functions by defining an appropriate inner product space. On the basis of the dual forward-backward splitting algorithm and the primal-dual forward-backward splitting algorithm, we develop several iterative algorithms that involve only computing the gradient of the differentiable function and proximity operators of related convex functions. These iterative algorithms are matrix-inversion-free and completely splitting algorithms. Finally, we employ the proposed iterative algorithms to solve a regularized general prior image constrained compressed sensing (PICCS) model that is derived from computed tomography (CT) image reconstruction under sparse sampling of projection measurements. Numerical results show that the proposed iterative algorithms outperform other algorithms.

\end{abstract}

% Include a list of keywords after the abstract

\noindent \textbf{Keywords:} Forward-backward splitting method; Primal-dual; Dual; Proximity operator.

 \noindent \textbf{2010 Mathematics Subject Classification:} 90C25;
 65K10.

\section{Introduction}
\label{sec:intro}  % \label{} allows reference to this section

In this paper, we consider solving a composite convex optimization problem that takes the form of
\begin{equation}\label{composite-sum}
\min_{x\in H}\ f(x) + g(x) + \sum_{i=1}^{m}h_i (B_i x),
\end{equation}
where $H$ and $\{G_i\}_{i=1}^{m}$ are Hilbert spaces, $f\in \Gamma_{0}(H)$ is differentiable with an $L$-Lipschitz continuous gradient for some constant $L\in (0,+\infty)$, and $g\in \Gamma_{0}(H)$ may be nonsmooth. Given an integer $m\geq 1$, for each $i\in \{1, 2, \cdots, m\}$,  $h_i \in \Gamma_{0}(G_i)$ and $L_i :H \rightarrow G_i$ is a bounded linear operator. Here, $\Gamma_{0}(H)$ denotes the class of proper lower semicontinuous (lsc) convex functions that are defined in the Hilbert space $H$. In the following, we always assume that the proximity operators with respect to $g$ and $\{h_i\}_{i=1}^{m}$ have a closed-form solution.  The optimization model (\ref{composite-sum}) includes a large number of existing models as special cases. For example,

(i)\ Let $m=1$. For simplicity and brevity, we drop the subscript "1". Then, the optimization problem (\ref{composite-sum}) reduces to
\begin{equation}\label{three-sum-intro}
\min_{x\in H}\ f(x) + g(x) + h(Bx),
\end{equation}
which has been studied in \cite{condat2013,lorenz2015JMIV,chenpj2016,yanm2016,tangyuchao2017}.

(ii)\ Let $f(x)=0$. Then, the optimization problem (\ref{composite-sum}) becomes
\begin{equation}\label{composite-sum-withoutdiff}
\min_{x\in H}\ g(x) + \sum_{i=1}^{m}h_i (B_i x),
\end{equation}
which has been studied in \cite{He2016IS,tang2016-1}.

(iii)\ Let $g(x) = 0$. Then, the optimization problem (\ref{composite-sum}) is equivalent to
\begin{equation}\label{composite-sum-withoutnondiff}
\min_{x\in H}\ f(x) + \sum_{i=1}^{m}h_i (B_i x),
\end{equation}
which has been studied in \cite{huangjz-2011}. Further, let $B_i = I$ for each $i\in \{1, 2, \cdots, m\}$, where $I$ denotes the identity operator. Then, the optimization problem (\ref{composite-sum-withoutnondiff}) reduces to
\begin{equation}
\min_{x\in H}\ f(x) + \sum_{i=1}^{m}h_i ( x),
\end{equation}
which has been studied in \cite{Raguet-SIAM-2013,Raduet2015SIAMJIS}.

(iv)\ Let $f(x)=0$ and $g(x)=0$. Then, the optimization problem (\ref{composite-sum}) reduces to
\begin{equation}\label{composite-simple}
\min_{x\in H}\ \sum_{i=1}^{m}h_i (B_i x),
\end{equation}
which has been studied in \cite{setzer2010}.

Owing to the emergence of the compressive sensing theory, the problem of minimizing the sum of two convex functions when $f=0$ or $g=0$ in (\ref{three-sum-intro}) has attracted considerable attention in recent years. A number of efficient iterative algorithms have been developed to solve such this problem, which  has wide application in signal and image processing. Examples include the  iterative shrinkage-thresholding algorithm (ISTA) and fast ISTA (FISTA) \cite{beck2009,beckandteboulle2009}, two-step ISTA \cite{bioucas2007}, primal-dual hybrid gradient algorithm (PDHGA) \cite{zhu2008}, primal-dual proximity algorithm (PDPA) \cite{chambolleandpock2011,pock1,chambolle2016AN} and primal-dual fixed point algorithm based on proximity operator (PDFP$^{2}$O) \cite{loris2011,chenpj-nanjing2013}, etc.

 Operator splitting is the most powerful methods for solving monotone inclusion problems, and it can be easily applied to the above-mentioned convex optimization problems. Operator splitting methods include forward-backward splitting \cite{lionsandmercier1979,passty1979JMAA,combettes2005}, Douglas-Rachford splitting \cite{bouglasandrachford1956TAMS,Eckstein1992} and forward-backward-forward splitting \cite{Tseng2000SIAM}. As operator splitting methods mainly focus on solving inclusion problems of the sum of two monotone operators  (see Definition \ref{def-monotone}), they cannot be directly used to solve composite convex optimization problems, such as (\ref{composite-sum}). A typical alternative is to transform the optimization problem or inclusion problem into a sum of two functions in a product space. Combettes and Pesquet \cite{combettes2012} first introduced a primal-dual splitting algorithm for solving monotone inclusions involving a mixture of sums, linear compositions, and parallel sums of set-valued and Lipschitz operators, including the composite convex optimization problem (\ref{composite-sum}) as a special case. Consequently, they obtained a primal-dual splitting algorithm to solve the optimization problem (\ref{composite-sum}), but the obtained iterative algorithm did not make full use of the cocoercive (see Definition \ref{def-cocoercive}) property of the gradient of function $f(x)$. In \cite{vu2013ACM}, Vu studied a monotone operator inclusion problem having the same formulation as that of \cite{combettes2012}. Under the condition that the involved main operator is cocoercive, Vu proposed a new type of primal-dual splitting algorithm to solve the composite convex optimization problem (\ref{composite-sum}). In contrast to the approach of Vu \cite{vu2013ACM}, Condat \cite{condat2013} proposed a primal-dual algorithm to solve the optimization problem (\ref{three-sum}) and extended it to solve the composite convex optimization problem (\ref{composite-sum}). In fact,  Vu's algorithm \cite{vu2013ACM} is equivalent to the Condat's algorithm \cite{condat2013}. Some generalizations of the Condat-Vu algorithm \cite{vu2013ACM,condat2013} can be found in \cite{liqia2016,wenmeng2016}.
  Combettes et al. \cite{combettes2014PICIP} further pointed out that the primal-dual version of the Condat-Vu algorithm \cite{vu2013ACM,condat2013} can be derived from the variable metric forward-backward splitting algorithm \cite{Combettes2014Optimization} framework. The variable metric forward-backward splitting algorithm provides a unified framework for analyzing the convergence of primal-dual splitting algorithms \cite{komodakis2015IEEE}.
  As the considered composite convex optimization problem (\ref{composite-sum}) is closely related to the optimization problem (\ref{three-sum-intro}) of the sum of three convex functions, besides the work of Vu \cite{vu2013ACM} and Condat \cite{condat2013},  we briefly review other existing iterative algorithms developed to solve it. Lorenz and Pock \cite{lorenz2015JMIV} proposed an inertial forward-backward splitting algorithm to solve the optimization problem (\ref{three-sum-intro}). This algorithm is derived from an inertial variable forward-backward splitting algorithm for finding a zero of the sum of two monotone operators, where one of the two operators is cocoercive.
   Motivated by the idea of the primal-dual fixed point algorithm based on proximity operator (PDFP$^{2}O$) \cite{chenpj-nanjing2013} and the preconditioned alternating projection algorithm (PAPA) \cite{krol2012IP}, Chen et al. \cite{chenpj2016} proposed the so-called primal-dual fixed point (PDFP) algorithm to solve the optimization problem (\ref{three-sum-intro}). They proved its convergence on the basis of the traditional fixed point theory. More recently, Yan \cite{yanm2016} proposed a new primal-dual algorithm for solving the optimization problem (\ref{three-sum-intro}), namely the primal-dual three operators (PD3O) algorithm. The PD3O algorithm can be reduced to the three operator splitting algorithm developed by Davis and Yin \cite{davis2015} when $B=I$ in (\ref{three-sum-intro}). Tang and Wu \cite{tangyuchao2017} proposed a general framework for solving the optimization problem (\ref{three-sum-intro}). The proposed iterative algorithms can recover the Condat-Vu algorithm \cite{condat2013,vu2013ACM}, the PDFP algorithm \cite{chenpj2016} and the PD3O algorithm \cite{yanm2016}. The key idea is to use two types of operator splitting methods: forward-backward splitting and three operator splitting. Although the obtained iterative schemes have a subproblem, which does not have a closed-form solution, they can be solved effectively by the dual and primal-dual approach.

In this paper, we propose several effective iterative algorithms to solve the composite convex optimization problem (\ref{composite-sum}). Although existing iterative algorithms, such as those of Vu \cite{vu2013ACM} and Condat \cite{condat2013} can be used to solve this problem, our approach offers the following novelties and improvements.  (1) To solve the optimization problem, we define the inner product space and transform the original optimization problem into that of the sum of three convex functions. By defining an inner product and a norm in the introduced inner product space, we analyze the proximity calculation of the corresponding function and the properties of an adjoint operator. We generalize the iterative algorithm proposed by Tang and Wu \cite{tangyuchao2017} to solve the composite convex optimization problem (\ref{composite-sum}). In addition, the convergence theorems of these algorithms are given. (2) The iterative parameters of the Condat-Vu algorithm \cite{vu2013ACM,condat2013} are controlled by an inequality that includes the Lipschitz constant and the operator norm. In our iterative algorithms, the Lipschitz constant and operator norm are independent. Thus, parameter selection is simplified. (3) To demonstrate the efficiency of our algorithm, we apply it to the regularized general prior image constrained compressed sensing (PICCS) model (\ref{regu-GPICCS}), which is derived from computed tomography (CT) image reconstruction. Furthermore, we compare the proposed iterative algorithms with some existing iterative algorithms, including the alternating direction method of multipliers (ADMM) \cite{boyd1}, the splitting primal-dual proximity algorithm \cite{tang2016-1}, and the preconditioned splitting primal-dual proximity algorithm \cite{tang2016-1}.

The remainder of the paper is organized as follows. Section \ref{basic_notation}  presents some basic notations and definitions used in this paper. Section \ref{existing-iterative}  reviews some existing iterative algorithms for solving the minimization problem (\ref{three-sum-intro}), which will be extended to solve the composite optimization problem (\ref{composite-sum}). Section \ref{three-operator-dual-primal-dual}  introduces our main iterative algorithms for solving the composite optimization problem (\ref{composite-sum}). In addition, the convergence of the proposed iterative algorithms is proved under mild conditions on the iterative parameters. Section \ref{numer_test} presents an application of the proposed iterative algorithms for solving the regularized general PICCS model. Further, numerical experiments on computed tomography (CT) image reconstruction are conducted to demonstrate the effectiveness and efficiency of the proposed iterative algorithms. In addition, the proposed iterative algorithms are compared with state-of-the-art methods, including the ADMM, splitting primal-dual proximity algorithm, and preconditioned splitting primal-dual proximity algorithm. Finally, Section \ref{conclusions} concludes the paper.

%%%%%%%%%%%%%%%%%%%%%%%%%%%%%%%%%%%%%%%%%%%%%%%%%%%%%%%%%%%%%%
\section{Preliminaries}
\label{basic_notation}
Throughout the paper, let $H$ be a Hilbert space, where $\langle \cdot, \cdot \rangle$ denotes the product defined on $H$ and its associated norm is $\|x\|=\sqrt{\langle x, x\rangle}$ for any $x\in H$. Let $\Gamma_{0}(H)$  denote the set of all proper lower semicontinuous convex functions from $H$ to $(-\infty, +\infty]$. Further, $\delta_{C}(x)$ denotes the indicator function that is $0$ if $x\in C$ and $+\infty$ otherwise. In addition, $I$ denotes the identity operator on $H$.
Let $A:H \rightarrow 2^{H}$ be a set-valued operator. The domain of $A$ is defined as $\emph{dom} A = \{x\in H | Ax \neq \emptyset \}$ and the graph of $A$ is defined as $\emph{gra } A = \{ (x,w)\in H\times H | u\in Ax \}$. Let $\emph{zer } A = \{ x\in H | 0\in Ax \}$ denote the set of zeros of $A$ and let $\emph{ran } A = \{u\in H | (\exists x\in H) u\in Ax\}$ denote the range of $A$. Let us recall the following concepts, which are commonly used in the context of convex analysis(see, for example, \cite{bauschkebook2011}).

\begin{definition}(Monotone operator and Maximal monotone operator)\label{def-monotone}
Let $A:H \rightarrow 2^{H}$ be a set-valued operator. Then, $A$ is monotone if
\begin{equation}
\langle x-y, u-v \rangle \geq 0, \quad \forall x,y\in H, \textrm{ and } u\in Ax, v\in Ay.
\end{equation}
Moreover, $A$ is maximal monotone if it is monotone and there exists no monotone $B:H \rightarrow 2^{H}$ such that $\emph{gra } B$ properly contains $\emph{gra } A$.
\end{definition}

\begin{definition}(Resolvent and Reflection operator)
The resolvent of $A$ is the operator $J_{A} = (I+A)^{-1}$. The reflection operator associated with $J_{A}$ is the operator $R_{A} = 2J_{A}-I$.

\end{definition}

\begin{definition}(Nonexpansive operator)
Let $T:H\rightarrow H$ be a single-valued operator. $T$ is nonexpansive if
\begin{equation}
\|Tx-Ty\| \leq \|x-y\|, \quad \forall x,y\in H.
\end{equation}
\end{definition}

\begin{definition}\label{def-cocoercive}
Let $T:H\rightarrow H$ be a single-valued operator. For any $\beta >0$, $T$ is $\beta$-cocoercive if
\begin{equation}
\langle x-y, Tx-Ty\rangle \geq \beta \|Tx-Ty\|^2, \quad \forall x,y\in H.
\end{equation}
\end{definition}

It follows from the Cauchy-Schwarz inequality that if $T$ is $\beta$-cocoercive, then $T$ is
$1/\beta$-Lipschitz continuous, but the converse is not true in general. However, from the Baillon-Haddad theorem, we know that
if $f:H\rightarrow R$ is a convex differential function with $1/\beta$-Lipschitz continuous gradient, then $\nabla f$ is  $\beta$-cocoercive.

The proximity operator is a natural generalization of the orthogonal projection operator, which was introduced by Moreau \cite{Moreau1962}. It plays an important role in the study of nonsmooth optimization problems.

\begin{definition}(Proximity operator)
Let $f\in \Gamma_{0}(H)$. For any $\lambda >0$, the proximity operator $prox_{\lambda f}$ is defined as
\begin{equation}
prox_{\lambda f} : H\rightarrow H\quad u \mapsto \arg\min_{x} \{ \frac{1}{2}\|x-u\|^2 + \lambda f(x)  \}.
\end{equation}
\end{definition}

Let $f\in \Gamma_{0}(H)$, where $\partial f$ is maximal monotone and the resolvent of $\partial f$ is the proximity operator of $f$, i.e., $prox_{f} = J_{\partial f}$.
It follows from the definition of the proximity operator that it is characterized by the inclusion $x = prox_{\lambda f}(u) \Leftrightarrow u-x \in \lambda \partial f(x)$. Furthermore, we have the inequality
 \begin{equation}\langle x-u, y-x \rangle \geq \lambda (f(x)-f(y)), \quad \forall y\in H,
 \end{equation}
  which is useful for proving the convergence of proximity algorithms. One of the most attractive properties of the proximity operator is that is firmly nonexpansive, i.e., also nonexpansive. Thus,
  \begin{equation}
  \| prox_{\lambda f}(x) - prox_{\lambda f}(y)  \|^2 \leq \| x-y \|^2 - \| (x-prox_{\lambda f}(x)) -(y-prox_{\lambda f}(y)) \|^2, \quad  \forall x,y\in H.
  \end{equation}
 Many simple convex functions which have a closed-form of proximity operator, such as the $\ell_1$-norm of $\| \cdot \|_1$ etc.
 The Moreau equality provides an alternative way to calculate the proximity operator from its Fenchel conjugate, i.e.,
 \begin{equation}
 prox_{\lambda f}(u) + \lambda prox_{\frac{1}{\lambda}f^*}(\frac{1}{\lambda}u)=u,
 \end{equation}
where $f^{*}$ denotes the Fenchel conjugate of $f$ and is defined as $f^{*}(y)=\max_{x}\langle x,y \rangle - f(x)$.
 Additional properties of the proximity operator and closed-form expressions of the proximity operator of various convex functions can be found in \cite{combettesbook2010}.

In the following, we briefly recall the discrete definition of total variation (TV). Let $u\in R^{n\times m}$. For example, $u$ could be an image. We assume reflexive boundary conditions for $u$ and define the discrete gradient $\nabla u = (D_{x}u, D_{y}u)$ as a forward difference operator:
\begin{equation}\label{}
D_{x}u(i,j) = \left\{
\begin{aligned}
& u(i,j+1)-u(i,j),  \textrm{ if } 1\leq i \leq n, 1\leq j < m, \\
& 0 , \textrm{ if } 1\leq i \leq n, j =m;
\end{aligned}
\right.
\end{equation}
and
\begin{equation}\label{}
D_{y}u(i,j) = \left\{
\begin{aligned}
& u(i+1,j)-u(i,j),  \textrm{ if } 1\leq i < n, 1\leq j \leq m, \\
& 0 , \textrm{ if } i=n, 1\leq j \leq m.
\end{aligned}
\right.
\end{equation}

Then, the so-called anisotropic total variation $\|u\|_{ATV}$ is defined as
\begin{equation}\label{atv}
\|u\|_{ATV} = \sum_{i=1}^{n}\sum_{j=1}^{m} (|(D_{x}u)_{(i,j)}| + | (D_{y}u)_{(i,j)} |).
\end{equation}
The isotropic total variation $\|u\|_{ITV}$ is defined as
\begin{equation}\label{itv}
\|u\|_{ITV} = \sum_{i=1}^{n}\sum_{j=1}^{m}\sqrt{ (D_{x} u)_{(i,j)}^{2} + (D_{y}u)_{(i,j)}^{2}  }.
\end{equation}

Next, we introduce another equivalent definition of the above-mentioned anisotropic total variation (\ref{atv}) and isotropic total variation (\ref{itv}). First, we need to transform the two-dimensional image matrix into a column vector. Let $u\in R^{nm}$, where $R^{nm}$ is the usual $nm$-dimensional Euclidean space. Let us define a first-order difference matrix $B$ as
\begin{equation}
B_{nm \times nm} = \left(
                       \begin{array}{ccccc}
                         -1 & 1 & 0 & \cdots & 0 \\
                         0 & -1 & 1 & \cdots & 0 \\
                          &  & \cdots &  &  \\
                         0 & 0 & \cdots & 0 & 0 \\
                       \end{array}
                     \right)
\end{equation}
and a matrix $D$ as
\begin{equation}\label{difference}
D = \left(
      \begin{array}{c}
        I \otimes B \\
        B \otimes I \\
      \end{array}
    \right),
\end{equation}
where $I$ denotes the identity matrix and $\otimes$ denotes the Kronecker inner product.

(1)\ The anisotropic total variation $\|u\|_{ATV}$ defined by (\ref{atv}) is equivalent to
$\|u\|_{ATV} = \|Du\|_{1}$.

(2)\  The isotropic total variation $\|u\|_{ITV}$ defined by (\ref{itv}) is equivalent to $\|u\|_{ITV} = \|Du\|_{2,1}$, where
$\|y\|_{2,1} = \sum_{i=1}^{nm}\sqrt{y_{i}^{2}+y_{nm+i}^{2}}$, $y\in R^{2nm}$.

In general, the performance of the anisotropic total variation is better than that of isotropic total variation in computed tomography (CT) image reconstruction(see, for example, \cite{tangyuchao20121,tangyuchao2013,tangyuchao2015}). Therefore, we report only the numerical results using anisotropic total variation in the numerical experiments.

%%%%%%%%%%%%%%%%%%%%%%%%%%%%%%%%%%%%%%%%%%%%%%%%%%%%%%%%%%%%%
\section{Iterative algorithms for solving the sum of three convex functions (\ref{three-sum})}
\label{existing-iterative}

In this section, we review some existing iterative algorithms for solving the minimization problem involving the sum of three convex functions in the form of
\begin{equation}\label{three-sum}
\min_{x\in H}\ f(x) + g(x) + h(Bx),
\end{equation}
where $H$ and $G$ are two Hilbert spaces, $f\in \Gamma_{0}(H)$ is differentiable with an $L$-Lipschitz continuous gradient for some $L \in (0,+\infty)$, $g\in \Gamma_{0}(H)$ and $h\in \Gamma_{0}(H)$ may be nonsmooth, and $B:H\rightarrow G$ is a bounded linear operator. The proximity operators of $g$ and $h$ are assumed to be easily computed.

In \cite{tangyuchao2017}, Tang et al. proposed a general framework for solving the optimization problem (\ref{three-sum}). The main idea is to combine the operator splitting methods with the dual and primal-dual solution of the subproblem. The following two main iterative algorithms were obtained.

\begin{algorithm}[H]
\caption{Dual forward-backward splitting algorithm for solving the optimization problem (\ref{three-sum})}
\begin{algorithmic}\label{dual-forward-backward}
\STATE \textbf{Initialize:}  Given arbitrary $x^{0}\in X$ and $y^{0}\in Y$. Choose $\gamma$ and $\lambda$.
\STATE 1. (Outer iteration step)\ For $k=0, 1, 2, \cdots$
\STATE \quad $u^{k} = x^k - \gamma \nabla f(x^k)$;
\STATE 2. (Inner iteration step)\ For $j_k = 0, 1, 2, \cdots$
\STATE \quad $y^{j_k +1} = prox_{\frac{\lambda}{\gamma}h^{*}} (y^{j_k}  + \frac{\lambda}{\gamma}B prox_{\gamma g}(u^{k} - \gamma B^{*}y^{j_k}) )$;
\STATE  End the inner iteration step when some stopping criterion is reached. Output: $y^{J_k +1}$.
\STATE 3. Update $x^{k+1} = prox_{\gamma g}(u^k - \gamma B^{*}y^{J_k +1})$;
\STATE 4. End the outer iteration step when some stopping criterion is reached.
\end{algorithmic}
\end{algorithm}

\begin{algorithm}[H]
\caption{Primal-dual forward-backward splitting algorithm for solving the optimization problem (\ref{three-sum})}
\begin{algorithmic}\label{primal-dual-forward-backward}
\STATE \textbf{Initialize:}  Given arbitrary $x^{0}, \overline{x}^{0}\in X$ and $y^{0}\in Y$. Choose $\gamma \in (0,2/L)$. Let $\sigma>0$ and $\tau >0$ satisfy the
condition $\tau \sigma < \frac{1}{\|B\|^2}$.
\STATE 1. (Outer iteration step)\ For $k=0, 1, 2, \cdots$
\STATE \quad $u^{k} = x^k - \gamma \nabla f(x^k)$;
\STATE 2. (Inner iteration step)\ For $j_k = 0, 1, 2, \cdots$
\STATE \quad 2.a. $\overline{x}^{j_k +1}  = prox_{\frac{\tau \gamma}{1+\tau}g}(\frac{\overline{x}^{j_k} - \tau B^{*}y^{j_k} + \tau u^k}{1+\tau})$;
\STATE \quad 2.b. $y^{j_{k}+1} = \gamma prox_{\frac{\sigma}{\gamma}h^{*}} (\frac{1}{\gamma}(y^{j_k} + \sigma B (2\overline{x}^{j_k +1} - \overline{x}^{j_k})) )$;
\STATE \ End the inner iteration when the primal-dual gap is less than some stopping criterion and output $\overline{x}^{J_k +1}$
\STATE 3. Update $x^{k+1} = \overline{x}^{J_k +1}$.
\STATE 4. End the outer iteration step when some stopping criterion is reached
\end{algorithmic}
\end{algorithm}

They proved the convergence of Algorithm \ref{dual-forward-backward} and Algorithm \ref{primal-dual-forward-backward} in finite-dimensional Hilbert spaces and obtained the following theorems.

\begin{thm}(\cite{tangyuchao2017})\label{tang-theorem1}
Let $\gamma \in (0,2/L)$ and $\lambda \in (0,2/\lambda_{max}(BB^*))$. For any $x^{0}\in X$ and $y^{0}\in Y$, the iterative sequences $\{x^k\}$ and $\{y^{j_k}\}$ are generated by Algorithm \ref{dual-forward-backward}. Then, the iterative sequence $\{x^k\}$ converges to a solution of the optimization problem (\ref{three-sum}).
\end{thm}

\begin{thm}(\cite{tangyuchao2017})\label{tang-theorem2}
Let $\gamma \in (0,2/L)$. Let $\sigma>0$ and $\tau >0$ satisfy the
condition that $\tau \sigma < \frac{1}{\|B\|^2}$. For any $x^{0}, \overline{x}^{0}\in X$ and $y^{0}\in Y$, the iterative sequences $\{x^k\}$, $\{\overline{x}^{k}\}$ and $\{y^{j_k}\}$ are generated by Algorithm \ref{primal-dual-forward-backward}.  Then, the iterative sequence $\{x^k\}$ converges to a solution of the optimization problem (\ref{three-sum}).
\end{thm}

Further, Tang et al. \cite{tangyuchao2017} showed out that Algorithm \ref{dual-forward-backward} and Algorithm \ref{primal-dual-forward-backward} could be reduced to the PDFP algorithm \cite{chenpj2016} and the Condat-Vu algorithm \cite{vu2013ACM,condat2013} for solving the optimization problem (\ref{three-sum}). In fact, let $j_k =k$ and let the number of inner iterations be $1$. Then, the iteration schemes of Algorithm \ref{dual-forward-backward} become
\begin{equation}\label{alg-PDFP}
 \left\{
\begin{aligned}
v^{k+1} & =  prox_{\gamma g}(x^k - \gamma \nabla f(x^k) - \gamma B^{*}y^{k}), \\
y^{k+1} & =  prox_{\frac{\lambda}{\gamma}h^{*}}(y^k + \frac{\lambda}{\gamma}Bv^{k+1}), \\
x^{k+1} & =  prox_{\gamma g}(x^k - \gamma \nabla f(x^k) - \gamma B^{*}y^{k+1}),
\end{aligned}
\right.
\end{equation}
which is exactly the PDFP algorithm proposed by Chen et al. \cite{chenpj2016}, who proved the convergence of (\ref{alg-PDFP}) under some conditions on the parameters $\gamma$ and $\lambda$.
\begin{thm}(\cite{chenpj2016})\label{pdfp-convergence}
Let $0< \gamma <2/L $ and $0< \lambda < 1/\lambda_{max}(BB^{*})$. Then, the iterative sequence $\{x^k\}$ generated by (\ref{alg-PDFP}) converges weakly to a solution of (\ref{three-sum}).
\end{thm}

\begin{rmk}
Chen et al. \cite{chenpj2016} proved the convergence of the iterative sequence (\ref{alg-PDFP}) in finite-dimensional Hilbert spaces. The weak convergence of the iterative sequence (\ref{alg-PDFP}) can be easily obtained by using the same technical proof of Chen et al. \cite{chenpj2016}.
\end{rmk}

Similarly, by letting $j_k =k$, $\overline{x}^{k}=x^k$, and the number of inner iterations be $1$, after simple calculation,
 the iteration schemes of Algorithm \ref{primal-dual-forward-backward} become
\begin{equation}\label{alg-condat3}
 \left\{
\begin{aligned}
x^{k+1} & =  prox_{\tau' g}(x^k - \tau' B^{*}y^k - \tau' \nabla f(x^k)), \\
y^{k+1} & =   prox_{\sigma' h^{*}}( y^k + \sigma' B(2x^{k+1} -x^k ) ).
\end{aligned}
\right.
\end{equation}
where $\sigma' = \frac{\sigma}{\gamma}$ and $\tau' = \frac{\tau \gamma}{1+\tau}$. The iteration scheme (\ref{alg-condat3}) was independently proposed by Condat \cite{condat2013} and Vu \cite{vu2013ACM}. The following convergence theorem was proved.
\begin{thm}\label{condat-convergence}(\cite{condat2013,vu2013ACM})
Let $\frac{1}{\tau'} -\sigma' \|B\|^2 > L/2$. Then, the iterative sequence $\{x^k\}$ generated by (\ref{alg-condat3}) converges weakly to a solution of (\ref{three-sum}).
\end{thm}

Besides Algorithm \ref{dual-forward-backward} and Algorithm \ref{primal-dual-forward-backward}, a dual three operator splitting algorithm and a primal-dual three operator splitting algorithm were also proposed in \cite{tangyuchao2017} to solve the optimization problem (\ref{three-sum}). In \cite{tangyuchao2017}, it was shown  that the dual three operator splitting algorithm recovered the PD3O algorithm presented in \cite{yanm2016}. Although the PDFP algorithm \cite{chenpj2016} and PD3O algorithm \cite{yanm2016} were obtained from different perspectives, they coincide with each other, as observed in \cite{tangyuchao2017}. The primal-dual three operator splitting algorithm is not the same as Algorithm \ref{primal-dual-forward-backward}, but numerical results presented in \cite{tangyuchao2017} showed that the performances of both algorithms are similar. Therefore, we conclude that Algorithm \ref{dual-forward-backward} and Algorithm \ref{primal-dual-forward-backward} are two basically iterative algorithms for solving the  optimization problem (\ref{three-sum}). In the next section, we will generalize them to solve the composite convex optimization problem (\ref{composite-sum}).

\section{Iterative algorithms for solving the composite convex optimization problem (\ref{composite-sum})}
\label{three-operator-dual-primal-dual}

To solve the composite convex optimization problem (\ref{composite-sum}), we need to group the composite function $\sum_{i=1}^{m}g_i (L_i x)$ into one function so that it can be simplified. First, let us introduce some notations.  Let us define Cartesian product space $\mathbf{G} = G_1 \times G_2 \times \cdots \times G_m$ and introduce two scalar products defined on $\mathbf{G}$.

$(1)$\ For any $\mathbf{y}=( y_1, \cdots, y_m )\in  \mathbf{G}$, $z = (z_1, \cdots, z_m)\in \mathbf{G}$, define
\begin{equation}\label{inner1}
\langle y,z \rangle_{1, \mathbf{G}} = \sum_{i=1}^{m}w_i \langle y_i, z_i\rangle_{G_i},
\end{equation}
where $\{w_i\}_{i=1}^{m}\in (0,1)$ such that $\sum_{i=1}^{m}w_i =1$.

$(2)$\ For any $y=( y_1, \cdots, y_m )\in  \mathbf{G}$, $z = (z_1, \cdots, z_m)\in \mathbf{G}$, define
\begin{equation}\label{inner2}
\langle y,z \rangle_{2, \mathbf{G}} = \sum_{i=1}^{m} \langle y_i, z_i\rangle_{G_i}.
\end{equation}

It is easy to prove that the product space $\mathbf{G}$ equipped with product $\langle \cdot, \cdot \rangle_{1, \mathbf{G}}$ and $\langle \cdot, \cdot \rangle_{2, \mathbf{G}}$ are Hilbert spaces. The associated norm of $\mathbf{G}$ with product $\langle \cdot, \cdot \rangle_{1, \mathbf{G}}$ is
$ \|y\|_{1, \mathbf{G}} = \sqrt{\sum_{i=1}^{m}w_i \|y_i\|_{G_i}^{2}} $ and the associated norm of $\mathbf{G}$ with product $\langle \cdot, \cdot \rangle_{2, \mathbf{G}}$ is
$ \|y\|_{2, \mathbf{G}} = \sqrt{\sum_{i=1}^{m} \|y_i\|_{G_i}^{2}} $.

Further, let us define the function $\mathbf{h}(y) = \sum_{i=1}^{m}h_i (y_i)$, for any $\mathbf{y}=( y_1, \cdots, y_m )\in  \mathbf{G}$, and the linear operator $\mathbf{B}: H \rightarrow \mathbf{G}$ by $\mathbf{B}(x) = (B_1 x, \cdots, B_m x)$. Then, the original optimization problem (\ref{composite-sum}) can be recast as follows,
\begin{equation}\label{comp-three-sum}
\min_{x\in H}\ f(x) + g(x) + \mathbf{h}(\mathbf{B}x).
\end{equation}

We note that the function $\mathbf{h}$ is proper, convex, and lower semi-continuous because each $h_i \in \Gamma_{0}(G_i)$ is proper, convex, and lower semicontinuous.
To derive an effective iterative algorithm, we prove some useful properties of the function $\mathbf{h}$ and linear operator $\mathbf{B}$.
The first lemma is proved under the product $\langle \cdot, \cdot \rangle_{1, \mathbf{G}}$ and norm $\|y\|_{1, \mathbf{G}}$ on $\mathbf{G}$.

\begin{lem}\label{our-lemma1}
Let $\mathbf{y} = (y_1, \cdots, y_m)\in \mathbf{G}$. For any $\lambda >0$, we have

$\emph{(i)}$\ $prox_{\lambda \mathbf{h}}(\mathbf{y}) = (prox_{\frac{\lambda}{w_1}h_1}(y_1), \cdots, prox_{\frac{\lambda}{w_m}h_m}(y_m))$;

$\emph{(ii)}$\ $\mathbf{h}^{*}(\mathbf{y}) = \sum_{i=1}^{m}h_{i}^{*}(w_i y_i)$ and $prox_{\lambda \mathbf{h}^{*}}(\mathbf{y}) = (\frac{1}{w_1}prox_{w_1 \lambda h_1}(w_1 y_1), \cdots, \frac{1}{w_m}prox_{w_m \lambda h_m}(w_m y_m))$;

$\emph{(iii)}$\ $\mathbf{B}^{*}(\mathbf{y}) = \sum_{i=1}^{m}w_i B_{i}^{*}y_i$.

\end{lem}

\begin{proof}
$\emph{(i)}$\ From the definition of the proximity operator, we have
\begin{align}
prox_{\lambda \mathbf{h}}(\mathbf{y}) & = \arg\min_{\mathbf{x}\in \mathbf{G}} \{ \frac{1}{2}\| \mathbf{x} -\mathbf{y} \|_{1, \mathbf{G}}^{2} + \lambda \mathbf{h}(\mathbf{x}) \} \nonumber \\
& = \arg\min_{\mathbf{x}\in \mathbf{G}} \{ \frac{1}{2} \sum_{i=1}^{m}w_i \|x_i - y_i\|_{G_i}^{2} + \lambda \sum_{i=1}^{m}h_{i}(x_i) \} \nonumber \\
& = (prox_{\frac{\lambda}{w_1}h_1}(y_1), \cdots, prox_{\frac{\lambda}{w_m}h_m}(y_m)).
\end{align}

$\emph{(ii)}$\ On the basis of the Fenchel conjugate, we have
\begin{align}
\mathbf{h}^{*}(\mathbf{y}) & = \sup_{\mathbf{x}} \langle \mathbf{x}, \mathbf{y} \rangle_{1, \mathbf{G}} - \mathbf{h}(\mathbf{x}) \nonumber \\
& = \sup_{\mathbf{x}} \sum_{i=1}^{m}w_i \langle x_i, y_i \rangle_{G_i} - \sum_{i=1}^{m}h_{i}(x_i) \nonumber \\
& = \sum_{i=1}^{m}h_{i}^{*}(w_i y_i).
\end{align}
Further, we obtain
\begin{align}
prox_{\lambda \mathbf{h}^{*}}(\mathbf{y}) & = \arg\min_{\mathbf{x}\in \mathbf{G}} \{ \frac{1}{2}\| \mathbf{x} -\mathbf{y} \|_{1, \mathbf{G}}^{2} + \lambda \mathbf{h}^{*}(\mathbf{x}) \} \nonumber \\
& = \arg\min_{\mathbf{x}\in \mathbf{G}} \{ \frac{1}{2} \sum_{i=1}^{m}w_i \|x_i - y_i\|_{G_i}^{2} + \lambda \sum_{i=1}^{m}h_{i}^{*}(w_i x_i) \} \nonumber \\
& = (\frac{1}{w_1}prox_{w_1 \lambda h_1}(w_1 y_1), \cdots, \frac{1}{w_m}prox_{w_m \lambda h_m}(w_m y_m)).
\end{align}

$\emph{(iii)}$\ It follows from $\langle \mathbf{B}x, \mathbf{y}\rangle_{1, \mathbf{G}} = \langle x, \mathbf{B}^{*}(\mathbf{y})\rangle$ that
\begin{align}
\langle x, \mathbf{B}^{*}(\mathbf{y}) \rangle & = \langle \mathbf{B}x, \mathbf{y}   \rangle_{1, \mathbf{G}} \nonumber \\
& = \sum_{i=1}^{m}w_i \langle B_i x, y_i \rangle \nonumber \\
& = \langle x, \sum_{i=1}^{m}w_i  B_{i}^{*}y_i \rangle,
\end{align}
which means that $\mathbf{B}^{*}(\mathbf{y}) = \sum_{i=1}^{m}w_i B_{i}^{*}y_i$.

\end{proof}

Lemma \ref{our-lemma1} presents the basic properties of $prox_{\lambda \mathbf{h}}$, $\mathbf{h}^{*}$, and $\mathbf{B}^{*}$ in the Hilbert space $\mathbf{G}$ equipped with product $\langle \cdot, \cdot \rangle_{1, \mathbf{G}}$ and norm $\|y\|_{1, \mathbf{G}}$. We can obtain similar results in the Hilbert space $\mathbf{G}$ equipped with product $\langle \cdot, \cdot \rangle_{2, \mathbf{G}}$ and norm $\|y\|_{2, \mathbf{G}}$.

\begin{lem}\label{our-lemma2}
Let $\mathbf{y} = (y_1, \cdots, y_m)\in \mathbf{G}$. For any $\lambda >0$, we have

$\emph{(i)}$\ $prox_{\lambda \mathbf{h}}(\mathbf{y}) = (prox_{\lambda h_1}(y_1), \cdots, prox_{\lambda h_m}(y_m))$;

$\emph{(ii)}$\ $\mathbf{h}^{*}(\mathbf{y}) = \sum_{i=1}^{m}h_{i}^{*}( y_i)$ and $prox_{\lambda \mathbf{h}^{*}}(\mathbf{y}) = (prox_{ \lambda h_1}(y_1), \cdots, prox_{\lambda h_m}( y_m))$;

$\emph{(iii)}$\ $\mathbf{B}^{*}(\mathbf{y}) = \sum_{i=1}^{m} B_{i}^{*}y_i$.

\end{lem}

As the technical proof process of Lemma \ref{our-lemma2} follows the same process as that of Lemma \ref{our-lemma1}, we omit it here. Furthermore, it is not difficult to prove that the operator norm of $\mathbf{B}$ is equal to $\sqrt{\sum_{i=1}^{m}w_i \|B_i\|^2}$ and $\sqrt{\sum_{i=1}^{m}\|B_i\|^2}$ under the product $\langle \cdot, \cdot \rangle_{1, \mathbf{G}}$ and $\langle \cdot, \cdot \rangle_{2, \mathbf{G}}$, respectively.

\subsection{Dual forward-backward type of splitting algorithm for solving (\ref{composite-sum})}

Now, we are ready to present the detailed iterative algorithms for solving the composite convex optimization problem (\ref{composite-sum}).
The first type of iterative algorithms is derived from the dual forward-backward splitting algorithm (Algorithm \ref{dual-forward-backward}).

\begin{algorithm}[H]
\caption{First class of dual forward-backward splitting algorithm for solving the composite convex optimization problem (\ref{composite-sum})}
\begin{algorithmic}\label{dual-forward-backward-composite1}
\STATE \textbf{Initialize:}  Given arbitrary $x^{0}\in H$ and $\mathbf{y}^{0}=(y_{1}^{0}, \cdots, y_{m}^{0})\in \mathbf{G}$. Choose $\gamma $ and $ \lambda $.
\STATE 1. (Outer iteration step)\ For $k=0, 1, 2, \cdots$
\STATE \quad $u^{k} = x^k - \gamma \nabla f(x^k)$;
\STATE 2. (Inner iteration step)\ For $j_k = 0, 1, 2, \cdots$
\STATE \quad  $y_{i}^{j_k +1} = \frac{1}{w_i}prox_{w_i \frac{\lambda}{\gamma}h_{i}^{*}}(w_i (y_{i}^{j_k} + \frac{\lambda}{\gamma}B_i prox_{\gamma g}(u^k - \gamma \sum_{i=1}^{m}w_i B_{i}^{*}y_{i}^{j_k})  ))$;
\STATE \ End the inner iteration when some stopping criteria is reached
\STATE 3. Update $x^{k+1} = prox_{\gamma g}(u^k - \gamma \sum_{i=1}^{m}w_i B_{i}^{*}y_{i}^{J_k +1})$.
\STATE 4. End the outer iteration step when some stopping criteria is reached
\end{algorithmic}
\end{algorithm}

\begin{algorithm}[H]
\caption{Second class of dual forward-backward splitting algorithm for solving the composite convex optimization problem (\ref{composite-sum})}
\begin{algorithmic}\label{dual-forward-backward-composite2}
\STATE \textbf{Initialize:}  Given arbitrary $x^{0}\in H$ and $\mathbf{y}^{0}=(y_{1}^{0}, \cdots, y_{m}^{0})\in \mathbf{G}$. Choose $\gamma$ and $ \lambda$.
\STATE 1. (Outer iteration step)\ For $k=0, 1, 2, \cdots$
\STATE \quad $u^{k} = x^k - \gamma \nabla f(x^k)$;
\STATE 2. (Inner iteration step)\ For $j_k = 0, 1, 2, \cdots$
\STATE \quad  $y_{i}^{j_k +1} = prox_{ \frac{\lambda}{\gamma}h_{i}^{*}}( y_{i}^{j_k} + \frac{\lambda}{\gamma}B_i prox_{\gamma g}(u^k - \gamma \sum_{i=1}^{m} B_{i}^{*}y_{i}^{j_k}  ))$;
\STATE \ End the inner iteration when some stopping criteria is reached
\STATE 3. Update $x^{k+1} = prox_{\gamma g}(u^k - \gamma \sum_{i=1}^{m} B_{i}^{*}y_{i}^{J_k +1})$.
\STATE 4. End the outer iteration step when some stopping criteria is reached
\end{algorithmic}
\end{algorithm}

We present the following convergence theorems related to Algorithm \ref{dual-forward-backward-composite1} and Algorithm \ref{dual-forward-backward-composite2}, respectively.

\begin{thm}\label{our-theorem12}
Let $\gamma \in (0,2/L)$ and $0< \lambda < 2/\sum_{i=1}^{m}w_i \|B_i\|^2$, where $\{w_i\}_{i=1}^{m}\subset (0,1)$ with $\sum_{i=1}^{m}w_i =1$. For any $x^{0}\in H$ and $\mathbf{y}^{0}=(y_{1}^{0}, \cdots, y_{m}^{0})\in \mathbf{G}$, where $\mathbf{G}$ is equipped with product $\langle \cdot, \cdot\rangle_{1,\mathbf{G}}$, let the iterative sequences $\{u^k\}$, $\{y_{i}^{j_k}\}_{i=1}^{m}$ and $\{x^k\}$ are generated by Algorithm \ref{dual-forward-backward-composite1}. Then, the iterative sequence $\{x^k\}$ converges to a solution of the composite optimization problem (\ref{composite-sum}).
\end{thm}

\begin{proof}
Since the optimization problem (\ref{composite-sum}) is equivalent to the optimization problem (\ref{comp-three-sum}), so we obtain the following iterative sequence for solving (\ref{composite-sum}) via Algorithm \ref{dual-forward-backward}.
\begin{equation}\label{3-eq4}
 \left\{
\begin{aligned}
u^{k} & = x^k - \gamma \nabla f(x^k),\\
\mathbf{y}^{j_k +1} & = prox_{\frac{\lambda}{\gamma}\mathbf{h}^{*}} ( \mathbf{y}^{j_k} + \frac{\lambda}{\gamma}\mathbf{B}prox_{\gamma g}(u^k - \gamma \mathbf{B}^* \mathbf{y}^{j_k})), \\
x^{k+1} & = prox_{\gamma g}(u^k - \gamma \mathbf{B}^{*}\mathbf{y}^{J_k}).
\end{aligned}
\right.
\end{equation}
By Theorem \ref{tang-theorem1}, we can conclude that the iterative sequence $\{x^k\}$ converges to a solution of the optimization problem (\ref{composite-sum}). With the help of Lemma \ref{our-lemma1}, we can split the iterative sequence (\ref{3-eq4}) and obtain the corresponding Algorithm \ref{dual-forward-backward-composite1} as stated before.

\end{proof}

\begin{thm}\label{our-theorem1}
Let $\gamma \in (0,2/L)$ and $0< \lambda < 1/\sum_{i=1}^{m}w_i \|B_i\|^2$, where $\{w_i\}_{i=1}^{m}\subset (0,1)$ with $\sum_{i=1}^{m}w_i =1$. For any $x^{0}\in H$ and $\mathbf{y}^{0}=(y_{1}^{0}, \cdots, y_{m}^{0})\in \mathbf{G}$, where $\mathbf{G}$ is equipped with product $\langle \cdot, \cdot\rangle_{1,\mathbf{G}}$, let the iterative sequences $\{u^k\}$, $\{y_{i}^{j_k}\}_{i=1}^{m}$ and $\{x^k\}$ are generated by Algorithm \ref{dual-forward-backward-composite1}. Let the number of inner iterations be $1$ and the updated iterative sequences $y_{i}^{J_k +1} = y_{i}^{k+1}$ for any $i\in \{1, 2, \cdots , m\}$.
Then, the iterative sequence $\{x^k\}$  converges weakly to a solution of the composite optimization problem (\ref{composite-sum}).
\end{thm}

\begin{proof}
By Lemma \ref{our-lemma1}, the iteration scheme of Algorithm \ref{dual-forward-backward-composite1} can be rewritten as follows:
\begin{equation}\label{3-eq1}
 \left\{
\begin{aligned}
v^{k} & = prox_{\gamma g}( x^k - \gamma \nabla f(x^k) - \gamma \mathbf{B}^{*}\mathbf{y}^{k} ),\\
\mathbf{y}^{k +1} & = prox_{\frac{\lambda}{\gamma}\mathbf{h}^{*}} ( \mathbf{y}^{k}  + \frac{\lambda}{\gamma}\mathbf{B}v^k ), \\
x^{k+1} & = prox_{\gamma g}(x^k - \gamma \nabla f(x^k) - \gamma \mathbf{B}^{*}\mathbf{y}^{k +1}),
\end{aligned}
\right.
\end{equation}
where the dual variables $\mathbf{y}^{k} = (y_{1}^{k}, \cdots, y_{m}^{k})\in \mathbf{G}$. The obtained iteration scheme (\ref{3-eq1}) is exactly the PDFP algorithm adopted to solve the problem (\ref{comp-three-sum}) of the sum of three convex functions. Therefore, it follows from Theorem \ref{pdfp-convergence} that the iterative sequence $\{x^k\}$ converges weakly to an optimal solution of (\ref{comp-three-sum}), which is also a solution of the composite convex optimization problem (\ref{composite-sum}).

\end{proof}

\begin{rmk}
Theorem \ref{our-theorem12} provides a larger selection of parameter $\lambda$ than Theorem \ref{our-theorem1}. However, Theorem \ref{our-theorem12} only shows the convergence of Algorithm \ref{dual-forward-backward-composite1} in finite-dimensional Hilbert spaces. While Theorem \ref{our-theorem1} proves that the iterative sequence converges weakly to a solution of the optimization problem (\ref{composite-sum}). In practice, we prefer to use Algorithm \ref{dual-forward-backward-composite1} under the conditions of Theorem \ref{our-theorem1}.

\end{rmk}

Let the Hilbert space $\mathbf{G}$ be equipped with product $\langle \cdot, \cdot\rangle_{2,\mathbf{G}}$ (\ref{inner2}), by Lemma \ref{our-lemma2}, the iterative sequence (\ref{3-eq4}) reduces to the iteration scheme in Algorithm \ref{dual-forward-backward-composite2}. Further, In Algorithm \ref{dual-forward-backward-composite2}, let the number of inner iterations be $1$ and the updated iterative sequences $y_{i}^{J_k +1} = y_{i}^{k+1}$ for any $i\in \{1, 2, \cdots , m\}$. According to Lemma \ref{our-lemma2}, we can rewrite the iteration scheme of Algorithm \ref{dual-forward-backward-composite2} in the same way as that of (\ref{3-eq1}). Therefore, we have the following convergence theorems of Algorithm \ref{dual-forward-backward-composite2}. As the proof is similar to that of Theorem \ref{our-theorem12} and Theorem \ref{our-theorem1}, we omit it here.

\begin{thm}
Let $\gamma \in (0,2/L)$ and $0< \lambda < 2/\sum_{i=1}^{m} \|B_i\|^2$. For any $x^{0}\in H$ and $\mathbf{y}^{0}=(y_{1}^{0}, \cdots, y_{m}^{0})\in \mathbf{G}$, where $\mathbf{G}$ is equipped with product $\langle \cdot, \cdot\rangle_{2,\mathbf{G}}$, let the iterative sequences $\{u^k\}$, $\{y_{i}^{j_k}\}_{i=1}^{m}$ and $\{x^k\}$ are generated by Algorithm \ref{dual-forward-backward-composite2}. Then, the iterative sequence $\{x^k\}$ converges to a solution of the composite optimization problem (\ref{composite-sum}).
\end{thm}

\begin{thm}
Let $\gamma \in (0,2/L)$ and $0< \lambda < 1/\sum_{i=1}^{m} \|B_i\|^2$. For any $x^{0}\in H$ and $\mathbf{y}^{0}=(y_{1}^{0}, \cdots, y_{m}^{0})\in \mathbf{G}$, where $\mathbf{G}$ is equipped with product $\langle \cdot, \cdot\rangle_{2,\mathbf{G}}$, let the iterative sequences $\{u^k\}$, $\{y_{i}^{j_k}\}_{i=1}^{m}$ and $\{x^k\}$ are generated by Algorithm \ref{dual-forward-backward-composite2}. Let the number of inner iterations be $1$ and the updated iterative sequences $y_{i}^{J_k +1} = y_{i}^{k+1}$ for any $i\in \{1, 2, \cdots , m\}$.
Then, the iterative sequence $\{x^k\}$  converges weakly to a solution of the composite optimization problem (\ref{composite-sum}).
\end{thm}

\begin{rmk}\label{remark1}
 In comparison with Algorithm \ref{dual-forward-backward-composite2}, Algorithm \ref{dual-forward-backward-composite1} provides a flexible way to choose weight vectors $\{w_i\}_{i=1}^{m}$ for calculating the proximity operator of each function $h_{i}^{*}$. It is observed that if the weight vectors $\{w_i\}_{i=1}^{m}$ are chosen to be the same, i.e., $w_1 = w_2 =\cdots = w_m$, then Algorithm \ref{dual-forward-backward-composite1} and Algorithm \ref{dual-forward-backward-composite2} are identical. In fact, let $w_1 = w_2 =\cdots = w_m = w$, $\overline{y}_{i}^{j_k +1} = w y_{i}^{j_k +1}$ and $\overline{\lambda} = w \lambda$. Then, the updated sequence $\{y_{i}^{j_k +1}\}$ in Algorithm \ref{dual-forward-backward-composite1} can be rewritten as $\overline{y}_{i}^{j_k +1} = prox_{\frac{\overline{\lambda}}{\gamma}h_{i}^{*}}( \overline{y}_{i}^{j_k} + \frac{\overline{\lambda}}{\gamma}B_i prox_{\gamma g}(u^k - \gamma \sum_{i=1}^{m}B_{i}^{*}\overline{y}_{i}^{j_k}) )$. Note that the range of parameter $\overline{\lambda}$ is the same as that of $\lambda$ in Algorithm \ref{dual-forward-backward-composite2}. Hence, we can conclude that Algorithm \ref{dual-forward-backward-composite1} and Algorithm \ref{dual-forward-backward-composite2} are equivalent to each other.
\end{rmk}

%\begin{rmk}
%When $m=1$ in the composite convex optimization problem (\ref{composite-sum}), Algorithm \ref{dual-forward-backward-composite1} and Algorithm \ref{dual-forward-backward-composite2} reduce to the iterative algorithm developed in \cite{tangyuchao2017} to solve the optimization problem (\ref{three-sum-intro}) of the sum of three convex functions.
%
%\end{rmk}

\subsection{Primal-dual forward-backward type of splitting algorithm for solving (\ref{composite-sum})}

The second type of iterative algorithms for solving the composite convex optimization problem (\ref{composite-sum}) is based on the primal-dual forward-backward splitting algorithm (Algorithm \ref{primal-dual-forward-backward}). The detailed algorithms are presented in Algorithm \ref{primal-dual-forward-backward-composite1} and Algorithm \ref{primal-dual-forward-backward-composite2}.

\begin{algorithm}[H]
\caption{First class of primal-dual forward-backward splitting algorithm for solving the composite convex optimization problem (\ref{composite-sum})}
\begin{algorithmic}\label{primal-dual-forward-backward-composite1}
\STATE \textbf{Initialize:}  Given arbitrary $x^{0}, \overline{x}^{0}\in H$ and $\mathbf{y}^{0}=(y_{1}^{0}, \cdots, y_{m}^{0})\in \mathbf{G}$. Choose $\gamma \in (0,2/L)$ and $\sigma, \tau >0$ such that $\sigma \tau < 1/\sum_{i=1}^{m}w_i \|B_i\|^2$, where $\{w_i\}_{i=1}^{m}\subset (0,1)$ with $\sum_{i=1}^{m}w_i =1$.
\STATE 1. (Outer iteration step)\ For $k=0, 1, 2, \cdots$
\STATE \quad $u^{k} = x^k - \gamma \nabla f(x^k)$;
\STATE 2. (Inner iteration step)\ For $j_k = 0, 1, 2, \cdots$
\STATE \quad 2.a.   $\overline{x}^{j_k +1} = prox_{\frac{\tau \gamma}{1+\tau}g}(\frac{\overline{x}^{j_k} - \tau \sum_{i=1}^{m}w_i B_{i}^{*}y_{i}^{j_k} + \tau u^k }{1+\tau})$;
\STATE \quad 2.b. $y_{i}^{j_k +1} = \frac{\gamma}{w_i}prox_{\frac{w_i \sigma}{\gamma}h_{i}^{*}}(\frac{w_i}{\gamma} (y_{i}^{j_k} + \sigma B_i (2\overline{x}^{j_k +1} - \overline{x}^{j_k} )  ))$;
\STATE \ End the inner iteration when the primal-dual gap is less than some stopping criterion and output $\overline{x}^{J_k +1}$
\STATE 3. Update $x^{k+1} = \overline{x}^{J_k +1}$.
\STATE 4. End the outer iteration step when some stopping criterion is reached.
\end{algorithmic}
\end{algorithm}

\begin{algorithm}[H]
\caption{Second class of primal-dual forward-backward splitting algorithm for solving the composite convex optimization problem (\ref{composite-sum})}
\begin{algorithmic}\label{primal-dual-forward-backward-composite2}
\STATE \textbf{Initialize:}  Given arbitrary $x^{0}, \overline{x}^{0}\in H$ and $\mathbf{y}^{0}=(y_{1}^{0}, \cdots, y_{m}^{0})\in \mathbf{G}$. Choose $\gamma \in (0,2/L)$ and $\sigma, \tau >0$ such that $\sigma \tau < 1/\sum_{i=1}^{m} \|B_i\|^2$.
\STATE 1. (Outer iteration step)\ For $k=0, 1, 2, \cdots$
\STATE \quad $u^{k} = x^k - \gamma \nabla f(x^k)$;
\STATE 2. (Inner iteration step)\ For $j_k = 0, 1, 2, \cdots$
\STATE \quad 2.a.   $\overline{x}^{j_k +1} = prox_{\frac{\tau \gamma}{1+\tau}g}(\frac{\overline{x}^{j_k} - \tau \sum_{i=1}^{m} B_{i}^{*}y_{i}^{j_k} + \tau u^k }{1+\tau})$;
\STATE \quad 2.b. $y_{i}^{j_k +1} = \gamma prox_{\frac{ \sigma}{\gamma}h_{i}^{*}}(\frac{1}{\gamma} (y_{i}^{j_k} + \sigma B_i (2\overline{x}^{j_k +1} - \overline{x}^{j_k} )  ))$;
\STATE \ End the inner iteration when the primal-dual gap is less than some stopping criterion and output $\overline{x}^{J_k +1}$
\STATE 3. Update $x^{k+1} = \overline{x}^{J_k +1}$.
\STATE 4. End the outer iteration step when some stopping criterion is reached.
\end{algorithmic}
\end{algorithm}

Let $w_1 = w_2 =\cdots = w_m = w$ in Algorithm \ref{primal-dual-forward-backward-composite1}. It follows from the same reason as that stated in Remark \ref{remark1} that Algorithm \ref{primal-dual-forward-backward-composite1} and Algorithm \ref{primal-dual-forward-backward-composite2} are identical. In the following, we prove the convergence of Algorithm \ref{primal-dual-forward-backward-composite1} and Algorithm \ref{primal-dual-forward-backward-composite2} under some conditions.

\begin{thm}\label{our-theorem32}
Let $\gamma \in (0,2/L)$ and $\sigma, \tau >0$ such that $\sigma \tau < 1/\sum_{i=1}^{m}w_i \|B_i\|^2$, where $\{w_i\}_{i=1}^{m}\subset (0,1)$ with $\sum_{i=1}^{m}w_i =1$. For any $x^{0}, \overline{x}^{0}\in H$ and $\mathbf{y}^{0}=(y_{1}^{0}, \cdots, y_{m}^{0})\in \mathbf{G}$, where $\mathbf{G}$ is equipped with product $\langle \cdot, \cdot\rangle_{1,\mathbf{G}}$, let iterative sequences $\{u^k\}$, $\{\overline{x}^{j_k}\}$, $\{y_{i}^{j_k}\}_{i=1}^{m}$ and $\{x^k\}$ are generated by Algorithm \ref{primal-dual-forward-backward-composite1}. Then, the iterative sequence $\{x^k\}$ converges to a solution of the composite optimization problem (\ref{composite-sum}).
\end{thm}

\begin{proof}
The following iterative sequence is obtained by applying Algorithm \ref{primal-dual-forward-backward} to the optimization problem (\ref{comp-three-sum}).
\begin{equation}\label{3-eq5}
 \left\{
\begin{aligned}
u^k & = x^k - \gamma \nabla f(x^k), \\
\overline{x}^{j_k+1} & = prox_{\frac{\tau \gamma}{1+\tau}g}( \frac{\overline{x}^{j_k} - \tau \mathbf{B}^{*}\mathbf{y}^{j_k} +\tau u^{j_k}}{1+\tau} ),\\
\mathbf{y}^{j_k +1} & = \gamma prox_{\frac{\sigma}{\gamma}\mathbf{h}^{*}} ( \frac{1}{\gamma}(\mathbf{y}^{j_k}  + \sigma \mathbf{B}(2\overline{x}^{j_k+1}-\overline{x}^{j_k}) )), \\
x^{k+1} & = \overline{x}^{J_k +1}.
\end{aligned}
\right.
\end{equation}
Since the optimization problem (\ref{comp-three-sum}) is equivalent to the optimization problem (\ref{composite-sum}), it follows from Theorem \ref{tang-theorem2} that the iterative sequence $\{x^k\}$ converges to a solution of the composite optimization problem (\ref{composite-sum}).

\end{proof}

\begin{thm}\label{our-theorem3}
Let $\gamma \in (0,2/L)$ and $\sigma, \tau >0$ such that $\sigma \tau < 1/\sum_{i=1}^{m}w_i \|B_i\|^2$, where $\{w_i\}_{i=1}^{m}\subset (0,1)$ with $\sum_{i=1}^{m}w_i =1$. For any $x^{0}, \overline{x}^{0}\in H$ and $\mathbf{y}^{0}=(y_{1}^{0}, \cdots, y_{m}^{0})\in \mathbf{G}$, where $\mathbf{G}$ is equipped with product $\langle \cdot, \cdot\rangle_{1,\mathbf{G}}$, let iterative sequences $\{u^k\}$, $\{\overline{x}^{j_k}\}$, $\{y_{i}^{j_k}\}_{i=1}^{m}$ and $\{x^k\}$ are generated by Algorithm \ref{primal-dual-forward-backward-composite1}. Let $j_k =k$ and the number of inner iterations be $1$. Assume that $\overline{x}^{0}=x^0$ and $\overline{x}^{j_k}=x^k$ for any $k \geq 1$. Then, the iterative sequence $\{x^k\}$  converges weakly to a solution of the composite optimization problem (\ref{composite-sum}).
\end{thm}

\begin{proof}
Based on the assumption, by Lemma \ref{our-lemma1}, the iteration scheme of Algorithm \ref{primal-dual-forward-backward-composite1} can be recast as follows,
\begin{equation}\label{3-eq2}
 \left\{
\begin{aligned}
x^{k+1} & = prox_{\frac{\tau \gamma}{1+\tau}g}( x^k - \frac{\tau \gamma}{1+\tau} \nabla f(x^k) - \frac{\tau}{1+\tau} \mathbf{B}^{*}\mathbf{y}^{k} ),\\
\mathbf{y}^{k +1} & = \gamma prox_{\frac{\sigma}{\gamma}\mathbf{h}^{*}} ( \frac{1}{\gamma}(\mathbf{y}^{k}  + \sigma \mathbf{B}(2x^{k+1}-x^k) )),
\end{aligned}
\right.
\end{equation}
where $\mathbf{y}^{k} = (y_{1}^{k},\cdots, y_{k}^{m})\in \mathbf{G}$. Let $\overline{\mathbf{y}}^{k}=\frac{1}{\gamma}\mathbf{y}^{k}$, $\tau' = \frac{\tau \gamma}{1+\tau}$, and $\sigma' = \frac{\sigma}{\gamma}$. Then, the iterative sequences (\ref{3-eq2}) can be rewritten as
\begin{equation}\label{3-eq3}
 \left\{
\begin{aligned}
x^{k+1} & = prox_{\tau' g}( x^k - \tau' \nabla f(x^k) - \tau' \mathbf{B}^{*}\overline{\mathbf{y}}^{k} ),\\
\overline{\mathbf{y}}^{k +1} & =  prox_{\sigma' \mathbf{h}^{*}} ( \overline{\mathbf{y}}^{k}  + \sigma' \mathbf{B}(2x^{k+1}-x^k) ).
\end{aligned}
\right.
\end{equation}
The iteration scheme (\ref{3-eq3}) is equivalent to the Condat and Vu algorithm adopted to solve the problem (\ref{comp-three-sum}) of the sum of three convex functions. It is easy to check that $1/\tau' - \sigma' \|\mathbf{B}\|^2 > L/2$. Therefore, we can conclude from Theorem \ref{condat-convergence} that the iterative sequence $\{x^k\}$ converges weakly to a solution of the composite convex optimization problem (\ref{composite-sum}).

\end{proof}

By Lemma \ref{our-lemma2}, we can also represent the iteration scheme of Algorithm \ref{primal-dual-forward-backward-composite2} in the same way as (\ref{3-eq5}) and (\ref{3-eq2}).  As with Theorem \ref{our-theorem32} and Theorem \ref{our-theorem3}, we obtain the following convergence theorems of Algorithm \ref{primal-dual-forward-backward-composite2}.

\begin{thm}
Let $\gamma \in (0,2/L)$ and $\sigma, \tau >0$ such that $\sigma \tau < 1/\sum_{i=1}^{m} \|B_i\|^2$. For any $x^{0}, \overline{x}^{0}\in H$ and $\mathbf{y}^{0}=(y_{1}^{0}, \cdots, y_{m}^{0})\in \mathbf{G}$, where $\mathbf{G}$ is equipped with product $\langle \cdot, \cdot\rangle_{2,\mathbf{G}}$, let iterative sequences $\{u^k\}$, $\{\overline{x}^{j_k}\}$, $\{y_{i}^{j_k}\}_{i=1}^{m}$ and $\{x^k\}$ are generated by Algorithm \ref{primal-dual-forward-backward-composite2}. Then, the iterative sequence $\{x^k\}$ converges to a solution of the composite optimization problem (\ref{composite-sum}).
\end{thm}

\begin{thm}
Let $\gamma \in (0,2/L)$ and $\sigma, \tau >0$ such that $\sigma \tau < 1/\sum_{i=1}^{m} \|B_i\|^2$. For any $x^{0}, \overline{x}^{0}\in H$ and $\mathbf{y}^{0}=(y_{1}^{0}, \cdots, y_{m}^{0})\in \mathbf{G}$, where $\mathbf{G}$ is equipped with product $\langle \cdot, \cdot\rangle_{2,\mathbf{G}}$, let iterative sequences $\{u^k\}$, $\{\overline{x}^{j_k}\}$, $\{y_{i}^{j_k}\}_{i=1}^{m}$ and $\{x^k\}$ are generated by Algorithm \ref{primal-dual-forward-backward-composite2}. Let $j_k =k$ and the number of inner iterations be $1$. Assume that $\overline{x}^{0}=x^0$ and $\overline{x}^{j_k}=x^k$ for any $k \geq 1$. Then, the iterative sequence $\{x^k\}$ converges weakly to a solution of the composite convex optimization problem (\ref{composite-sum}).
\end{thm}

\begin{rmk}

(1) Compared with the Condat-Vu algorithm \cite{condat2013,vu2013ACM} for solving the composite optimization problem (\ref{composite-sum}), our algorithms, namely Algorithm \ref{primal-dual-forward-backward-composite1} and Algorithm \ref{primal-dual-forward-backward-composite2}, provide a simple way to select the iterative parameters.

(2) Let $f(x)=0$ in the iteration scheme (\ref{3-eq3}) of Algorithm \ref{primal-dual-forward-backward-composite1}. Then, it recovers the iterative algorithm proposed in \cite{He2016IS} for solving the optimization problem (\ref{composite-sum-withoutdiff}). The iteration scheme (\ref{3-eq3}) of Algorithm \ref{primal-dual-forward-backward-composite2} recovers the iterative algorithm proposed in \cite{tang2016-1}.

(3)
Huang et al. \cite{huangjz-2011} proposed two efficient iterative algorithms, namely the composite splitting algorithm (CSA) and Fast CSA, to solve the optimization problem (\ref{composite-sum-withoutnondiff}), where they assumed that $\{B_i\}_{i=1}^{m}$ are orthogonal matrices. Our iterative algorithms, namely Algorithm \ref{dual-forward-backward-composite1}, Algorithm \ref{dual-forward-backward-composite2}, Algorithm \ref{primal-dual-forward-backward-composite1} and Algorithm \ref{primal-dual-forward-backward-composite2} can be easily applied to solve the same optimization problem (\ref{composite-sum-withoutnondiff}) by letting $g(x)=0$. We discard the requirement that $\{B_i\}_{i=1}^{m}$ are orthogonal matrices.

(4)
Setzer et al. \cite{setzer2010} proposed an iterative algorithm for solving the optimization problem (\ref{composite-simple}), which is based on the ADMM. It includes a subproblem of solving a linear system equation that can be solved explicitly or iteratively.
  Our proposed iterative algorithms have a more simple structure. Every step of our proposed iterative algorithms has an explicit formulation.

\end{rmk}

%%%%%%%%%%%%%%%%%%%%%%%%%%%%%%%%%%%%%%%%%%%%%%%%%%%%%%%%%%%%%
\section{Numerical experiments}
\label{numer_test}
In this section, we will employ the proposed iterative algorithms to solve an optimization model that is suitable for reconstructing computed tomography (CT) images from a set of undersampled and potentially noisy projection measurements. The proposed iterative algorithms are compared with several representative algorithms, including ADMM \cite{boyd1}, splitting primal-dual proximity algorithm (SPDP) \cite{tang2016-1}, and preconditioned SPDP (Pre-SPDP) \cite{tang2016-1}. All the experiments are accomplished by Matlab and on a standard Lenovo laptop with Intel(R) Core(TM) i7-4712MQ 2.3GHz CPU and 4GB RAM.

First,  we briefly present an optimization model of CT image reconstruction.

\subsection{Computed tomography image reconstruction problem}

In recent years, the compressive sensing theory and related methods have been successfully applied to low-dose computed tomography (CT) image reconstruction problems. As most CT images have piece-wise smooth constants, the total variation (TV) regularization term has gained wide spread use in the CT image reconstruction model. Many efficient first-order methods have been developed to solve TV image reconstruction models.
In 2008, Chen et al. \cite{chengh2008} proposed a prior image constrained compressed sensing (PICCS) model for reconstructing  CT images from few-view and limited angles data. The PICCS model takes the form of
\begin{equation}\label{PICCS}
\begin{aligned}
\min_{x}\ & \alpha \|D_{1}(x-x_p)\|_{1} + (1-\alpha)\|D_2 x\|_1,\\
s.t. & Ax = b,
\end{aligned}
\end{equation}
where $D_1 \in R^{m_1 \times n}$ and $D_2\in R^{m_2 \times n}$ could be any sparsifying transforms, $\alpha \in [0,1]$, $A\in R^{m\times n}$ is the system matrix, $b\in R^{m\times 1}$ is the measured data, and $x_p \in R^{n}$ denotes a prior image. The PICCS model is a generalization of the traditional constrained total variation (TV) image reconstruction model. It reduces to the constrained TV model when $\alpha =0$. The PICCS model has attracted considerable attention because it requires less sampling data than the traditional constrained TV model to obtain a reasonable reconstruction image(see, for example, \cite{chengh2010} and the references therein. Recently, Tang and Zong \cite{tangyc20171} proposed a general PICCS model as follows:
\begin{equation}\label{general-PICCS}
\begin{aligned}
\min_{x}\ & \alpha \varphi_{1}(D_{1}(x-x_p)) + (1-\alpha)\varphi_2 (D_2 x), \\
s.t. & Ax = b, \\
& x\in C,
\end{aligned}
\end{equation}
where $\varphi_1$ and $\varphi_2$ are two proper, lower semicontinuous convex functions, $C$ is a nonempty closed convex set and the remaining assumptions are the same as those of (\ref{PICCS}) In comparison with the original PICCS model (\ref{PICCS}), the general PICCS model (\ref{general-PICCS}) has two advantages. (1) The $\ell_1$ norm in the PICCS model (\ref{PICCS}) is replaced by the general convex functions $\varphi_1$ and $\varphi_2$, which makes the model more flexible. (2) By introducing a closed convex set constraint, it can contain a priori information of the image to be reconstructed, such as nonnegative or bounded.

In this paper, we introduce a regularized general PICCS model as follows:
\begin{equation}\label{regu-GPICCS}
\begin{aligned}
\min_{x}\ & \frac{1}{2}\|Ax-b\|_{2}^{2} + \lambda_1 \varphi_{1}(D_{1}(x-x_p)) + \lambda_2 \varphi_2 (D_2 x), \\
s.t. \ & x\in C,
\end{aligned}
\end{equation}
where $\lambda_1 >0$ and $\lambda_2 >0$ are regularization parameters. The quadratic loss function is reasonable under the assumption that the collected data vector $b$ is corrupted by Gaussian noise. The regularized general PICCS model (\ref{regu-GPICCS}) can be rewritten as
\begin{equation}\label{regu-GPICCS2}
\min_{x}\  \frac{1}{2}\|Ax-b\|_{2}^{2} + \lambda_1 \varphi_{1}(D_{1}(x-x_p)) + \lambda_2 \varphi_2 (D_2 x) + \delta_{C}(x).
\end{equation}

The optimization problem (\ref{regu-GPICCS2}) is a special case of the composite convex optimization problem (\ref{composite-sum}). In fact, let $f(x) = \frac{1}{2}\|Ax-b\|_{2}^{2}$, $g(x) = \delta_{C}(x)$, $h_{1}(x) = \lambda_1 \varphi_{1}(x-D_1 x_p)$, $h_2 (x)=\lambda_2 \varphi_{2}(x)$, $B_1 = D_1$ and $B_2 = D_2$. To apply the proposed iterative algorithms, we have $\nabla f(x) = A^{T}(Ax-b)$ and the proximity operator of $g(x)$ is the orthogonal projection onto the closed convex set $C$, i.e., $prox_{\lambda g}(u)=P_{C}(u)$ for any $\lambda >0$. Furthermore, owing to the Moreau equality, the proximity operators of $h_{1}^{*}$ and $h_{2}^{*}$ can be calculated from $h_1$ and $h_2$, respectively. Therefore, the proposed iterative algorithms can be implemented  easily.

On the other hand, the optimization problem (\ref{regu-GPICCS2}) can also be rewritten in the form of (\ref{composite-sum-withoutdiff}). In fact, let $h_{1}(x) = \frac{1}{2}\|x-b\|_{2}^{2}$, $h_2 (x) = \lambda_1 \varphi_{1}(x-D_1 x_p)$, $h_{3}(x) = \lambda_2 \varphi_2 (x)$, $B_1 = A$, $B_2 = D_1$, $B_3 = D_2$ and $g(x) = \delta_{C}(x)$. Therefore, the splitting primal-dual proximity algorithm and the preconditioned splitting primal-dual proximity algorithm developed in \cite{tang2016-1} can be used to solve the regularized general PICCS model (\ref{regu-GPICCS}). In addition, the ADMM \cite{boyd1} provides a very general framework for solving various convex optimization problems including the regularized general PICCS model (\ref{regu-GPICCS}). We briefly present the unscaled form of ADMM method \cite{boyd1}  below for solving the regularized general PICCS model (\ref{regu-GPICCS}):
\begin{equation}\label{admm-algorithm}
 \left\{
\begin{aligned}
x^{k+1} & = \arg\min_{x\in C}\ \{ \frac{1}{2}\|Ax-b\|_{2}^{2} + \frac{\rho_1}{2} \| D_1 x - y_{1}^{k} + v_{1}^{k} \|^2 + \frac{\rho_2}{2} \| D_2 x - y_{2}^{k} + v_{2}^{k} \|^2   \},\\
y_{1}^{k +1} & = \arg\min_{y_1}\ \{ \lambda_1 \varphi_1 (y_1 - D_1 x_p) + \frac{\rho_1}{2}\| D_1 x^{k+1} - y_1 + v_{1}^{k} \|^2 \}, \\
y_{2}^{k +1} & = \arg\min_{y_2}\ \{ \lambda_2 \varphi_2 (y_2) + \frac{\rho_2}{2}\| D_2 x^{k+1} - y_2 + v_{2}^{k} \|^2 \}, \\
v_{1}^{k+1} & = v_{1}^{k} + D_1 x^{k+1} - y_{1}^{k+1}, \\
v_{2}^{k+1} & = v_{2}^{k} + D_2 x^{k+1} - y_{2}^{k+1},
\end{aligned}
\right.
\end{equation}
where $\rho_1 >0$ and $\rho_2 >0$. Here, the updated sequence $\{x^{k+1}\}$ is obtained by the standard gradient projection algorithm, i.e., $x^{k+1}= P_{C}(x^k - \gamma ( A^{T}(Ax^k -b) + \rho_1 D_{1}^{T}(D_1 x^k - y_{1}^{k} + v_{1}^{k}) + \rho_2 D_{2}^{T}(D_2 x^k - y_{2}^{k} + v_{2}^{k})))$, where $\gamma \in (0, 2/(\|A\|^2 + \rho_1 \|D_1\|^2 + \rho_2 \|D_2\|^2))$.

\subsection{Experimental setup}

The proposed iterative algorithms, namely Algorithm \ref{dual-forward-backward-composite2} and Algorithm \ref{primal-dual-forward-backward-composite2} are compared with ADMM, SPDP and Pre-SPDP. The parameters of each algorithm are set as follows. We set $\gamma = 1.9/(\|A\|^2 + \rho_1 \|D_1\|^2 + \rho_2 \|D_2\|^2)$  and $\rho_1 = \rho_2 =1$ for ADMM. For SPDP and Pre-SPDP, the parameters are set to be the same as those in \cite{tang2016-1}. We select Algorithm \ref{dual-forward-backward-composite2} and Algorithm \ref{primal-dual-forward-backward-composite2} as representatives of the proposed iterative algorithms. Algorithm \ref{dual-forward-backward-composite1} is equivalent to Algorithm \ref{dual-forward-backward-composite2} when the weight vectors are chosen to be the same. The same is true for Algorithm \ref{primal-dual-forward-backward-composite1} and Algorithm \ref{primal-dual-forward-backward-composite2}. For
 Algorithm \ref{dual-forward-backward-composite2} and Algorithm \ref{primal-dual-forward-backward-composite2}, we set $\gamma = 1.9/\|A\|^2$ and $\lambda = 0.9/(\|D_1\|^2 + \|D_2\|^2)$. Moreover, the parameter $\tau$ is set to $1$ for Algorithm \ref{primal-dual-forward-backward-composite2}. Total variation regularization has been proved to be very useful for  CT image reconstruction from limited angles and sparse view projections. Hence, for the regularized general PICCS model (\ref{regu-GPICCS}), we choose the regularization terms $\varphi_1(D_1 x - D_1 x_p)$ and $\varphi_2(D_2 x)$ as the total variation. Then, the matrices $D_1$ and $D_2$ are the first-order difference matrices, which are defined by (\ref{difference}). The functions $\varphi_1$ and $\varphi_2$ are the usual $\ell_1$-norm.

We use the same data-set as that used in \cite{tangyc20171} in the numerical experiments. The standard $256\times 256$ Shepp-Logan phantom image is used as the reconstructed image (Figure \ref{shepp-logan} (a)). The priori image $x_p$ in the regularized general PICCS model (\ref{regu-GPICCS}) is obtained by adding a random Gaussian noise with  mean $0$ and  variance $0.01$ to the original Shepp-Logan phantom (Figure \ref{shepp-logan} (b)). We set the nonempty closed convex set $C$ to be nonnegative set, i.e., $C = \{ x | x\geq 0\}$.
%-------------
   \begin{figure}
      \setlength{\abovecaptionskip}{-20pt}
   \begin{center}
   \begin{tabular}{c}
   \scalebox{0.65}{\includegraphics{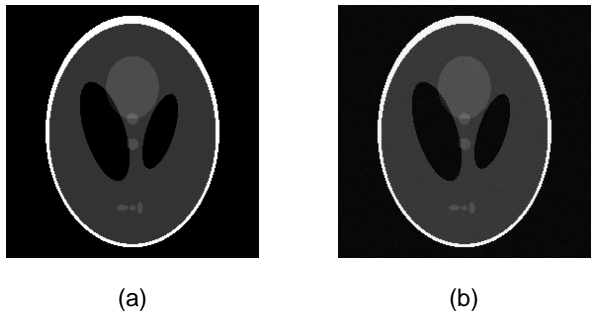}}
   \end{tabular}
   \end{center}
   \caption[]{(a)\ Standard $256\times 256$ Shepp-Logan phantom (b)\ Simulated priori image $x_p$. The $2$-norm error between the standard Shepp-Logan phantom and the priori image $x_p$ is $2.5529$. }
%>>>> use \label inside caption to get Fig. number with \ref{}
   { \label{shepp-logan}}
   \end{figure}

The phantom is scanned by a fan beam with $20$ views distributed from $0$ to $360$ and $100$ rays in each view. The size of the system matrix $A$ is $2000\times 65536$. The simulated projection data is generated by AIRtools \cite{hansen1}. Random Gaussian noise with mean $0$ and variance $e$ is added to the collected data vector $b$.

The image reconstruction quality is evaluated on the basis of the signal-to-noise ratio (SNR) and normalized mean square distance (NMSD), where
$$
SNR = 10log\frac{\|x-\overline{x}\|^2}{\|x_r - x\|^2},
$$
and
$$
NMSD = \frac{\|x-x_r\|}{\|x-\overline{x}\|},
$$
where $\overline{x}$ is the average value of the original image $x$ and $x_r$ is the reconstructed image.

We set the regularization parameters $\lambda_1 = 0.4$ and $\lambda_2 = 0.5$ throughout the test. The stopping criterion for all the test algorithms is defined as
$$
\frac{\|x^{k+1}-x^k\|}{\|x^k\|} < \epsilon,
$$
where $\epsilon >0$ is a given tolerance value.

\subsection{Numerical results and discussions}
In this subsection, we report the numerical results obtained by the considered iterative algorithms. The numerical results are summarized in Table \ref{results1}.
"Iter" represents the iteration number used when the iterative algorithm stops under the given stopping criterion $\epsilon$. The symbol '-' implies that the maximum number of iterations exceeds $40000$ .

\begin{table}[htbp]
\scriptsize
\centering
\caption{Numerical results obtained by the tested iterative algorithms for CT image reconstruction.}
\begin{tabular}{c|c|ccccccc}
\hline
Noise & \multirow{2}[1]{*}{Methods} &  \multicolumn{3}{c}{$ \epsilon = 10^{-6}$} &  & \multicolumn{3}{c}{$ \epsilon = 10^{-8}$} \\ \cline{3-5} \cline{7-9}
Level $e$ & &  $SNR(dB)$ & $NMSD$ & $Iter$ & & $SNR(dB)$ & $NMSD$ & $Iter$  \\
\hline
\hline
\multirow{5}[1]{*}{$0.01$}  & Algorithm \ref{dual-forward-backward-composite2} &  $38.7929$ & $0.0115$ & $2816$ & &  $38.8477$ & $0.0114$ & $5203$ \\
 & Algorithm \ref{primal-dual-forward-backward-composite2} &  $38.7043$ & $0.0116$ & $5230$ & & $38.8473$ & $0.0114$ & $9050$  \\
 & ADMM \cite{boyd1}&  $38.8478$ & $0.0114$ & $5674$ & &  $38.8483$ & $0.0114$ & $33029$  \\
 & SPDP \cite{tang2016-1} &  $38.7560$ & $0.0115$ & $-$ & &  $38.7560$ & $0.0115$ & $-$ \\
  & Pre-SPDP \cite{tang2016-1} &  $38.8214$ & $0.0115$ & $7937$ & &  $38.8469$ & $0.0114$ & $-$ \\
\hline
\multirow{5}[1]{*}{$0.05$}  & Algorithm \ref{dual-forward-backward-composite2} &  $38.2471$ & $0.0122$ & $2874$ & &  $38.3281$ & $0.0121$ & $5165$ \\
 & Algorithm \ref{primal-dual-forward-backward-composite2} &  $38.1596$ & $0.0124$ & $5339$ & & $38.3270$ & $0.0121$ & $9586$  \\
 & ADMM \cite{boyd1}&  $38.3539$ & $0.0121$ & $5682$ & &  $38.3275$ & $0.0121$ & $26180$  \\
 & SPDP \cite{tang2016-1}&  $38.3411$ & $0.0121$ & $-$ & &  $38.3411$ & $0.0121$ & $-$ \\
  & Pre-SPDP \cite{tang2016-1}&  $38.3476$ & $0.0121$ & $8833$ & &  $38.3286$ & $0.0121$ & $-$ \\
\hline
\multirow{5}[1]{*}{$0.1$}  & Algorithm \ref{dual-forward-backward-composite2} &  $36.5793$ & $0.0148$ & $2957$ & &  $36.6511$ & $0.0147$ & $5898$ \\
 & Algorithm \ref{primal-dual-forward-backward-composite2} &  $36.4905$ & $0.0150$ & $5376$ & & $36.6500$ & $0.0147$ & $10825$  \\
 & ADMM \cite{boyd1}&  $36.6695$ & $0.0147$ & $5896$ & &  $36.6526$ & $0.0147$ & $25855$  \\
 & SPDP \cite{tang2016-1}&  $36.6684$ & $0.0147$ & $-$ & &  $36.6684$ & $0.0147$ & $-$ \\
  & Pre-SPDP \cite{tang2016-1}&  $36.6677$ & $0.0147$ & $8546$ & &  $36.6522$ & $0.0147$ & $-$ \\
\hline
\end{tabular}\label{results1}
\end{table}

We can see from Table \ref{results1} that Algorithm \ref{dual-forward-backward-composite2} requires fewer iterations than the other iterative algorithms to achieve the same accuracy. The SNR value achieved by Algorithm \ref{primal-dual-forward-backward-composite2} is nearly the same as that achieved by Algorithm \ref{dual-forward-backward-composite2}. However, Algorithm \ref{primal-dual-forward-backward-composite2} requires more number of iterations than the Algorithm \ref{dual-forward-backward-composite2}. The SNR and function value versus number of iterations of all the tested algorithms are shown in Figure \ref{snrandfvalue}. We can observe that the objective function value of all the tested algorithms reaches nearly the same minimum value. For better visual inspection, Figure \ref{snrandfvaluezoomin} shows magnified views of Figure \ref{snrandfvalue}.  We observe that the SNR values of Algorithm \ref{dual-forward-backward-composite2} and Algorithm \ref{primal-dual-forward-backward-composite2} always increase with the number of iterations. However, the SNR values of the ADMM, SPDP and Pre-SPDP increase first, and then decrease after reaching the maximum value. For example, in the experiment with $e=0.01$ and $\epsilon = 10^{-8}$, the maximum SNR value and the corresponding number of iterations are $38.9533$ (dB) and $2398$ for ADMM.  $38.7777$ (dB) and  $38706$ for SPDP, and $38.8469$ (dB) and  $39381$ for Pre-SPDP.
Therefore, we can conclude that Algorithm \ref{dual-forward-backward-composite2} and Algorithm \ref{primal-dual-forward-backward-composite2} are more robust than the other algorithms in CT image reconstruction. Figure \ref{reconstructedimages} shows the reconstructed images obtained by the tested algorithms.

%-------------
   \begin{figure}
      \setlength{\abovecaptionskip}{-20pt}
   \begin{center}
   \begin{tabular}{c}
   \scalebox{0.65}{\includegraphics{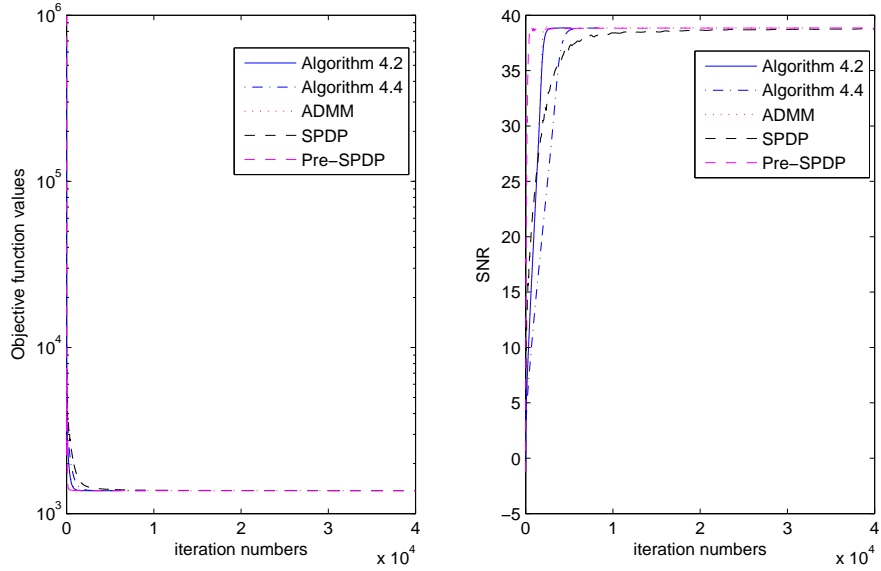}}
   \end{tabular}
   \end{center}
   \caption[]{Comparison of the tested algorithms when $e=0.01$ and $\epsilon = 10^{-8}$. (a) Objective function values versus number of iterations and (b) SNR versus number of iterations. }
%>>>> use \label inside caption to get Fig. number with \ref{}
   { \label{snrandfvalue}}
   \end{figure}

%-------------
   \begin{figure}
      \setlength{\abovecaptionskip}{-20pt}
   \begin{center}
   \begin{tabular}{c}
   \scalebox{0.65}{\includegraphics{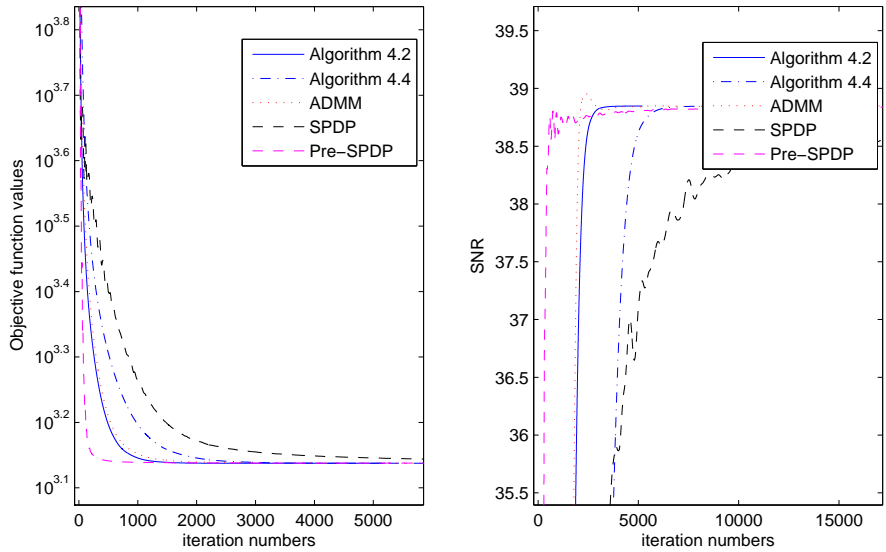}}
   \end{tabular}
   \end{center}
   \caption[]{Magnified views of Figure \ref{snrandfvalue}.  }
%>>>> use \label inside caption to get Fig. number with \ref{}
   { \label{snrandfvaluezoomin}}
   \end{figure}

%-------------
   \begin{figure}
      \setlength{\abovecaptionskip}{-20pt}
   \begin{center}
   \begin{tabular}{c}
   \scalebox{0.65}{\includegraphics{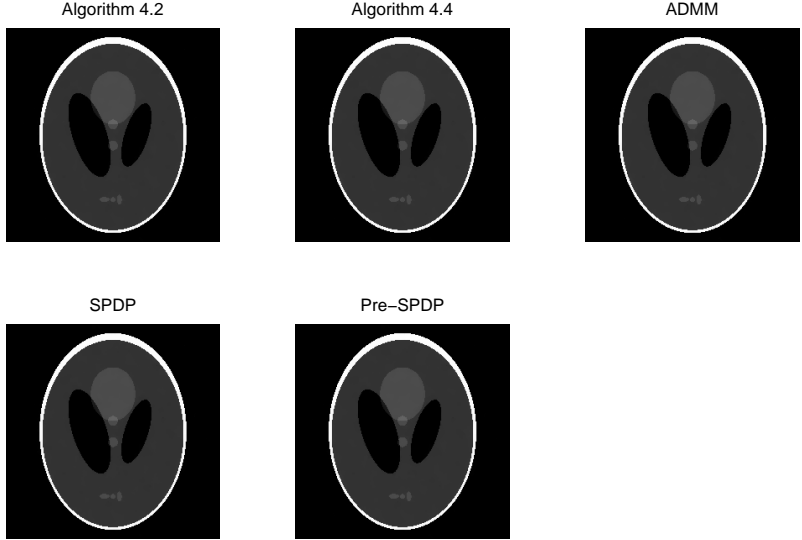}}
   \end{tabular}
   \end{center}
   \caption[]{Reconstructed results obtained by the tested algorithms with $e=0.1$ and $\epsilon = 10^{-8}$. }
%>>>> use \label inside caption to get Fig. number with \ref{}
   { \label{reconstructedimages}}
   \end{figure}

%%%%%%%%%%%%%%%%%%%%%%%%%%%%%%%%%%%%%%%%%%%%%%%%%%%%%%%%%%%%%
\section{Conclusions}
\label{conclusions}
In this paper, we proposed several first-order splitting algorithms for solving the composite convex optimization problem (\ref{composite-sum}). Many problems arising in signal recovery and image processing are  special cases of this problem. To solve the considered convex optimization problem (\ref{composite-sum}), we defined a Cartesian product space and transformed the original optimization problem into that of the sum of three convex functions, including a smooth function with a Lipschitz continuous gradient, a nonsmooth proximable function, and a composition of a proximable function and a linear operator. In the newly defined inner product space, we define an inner product and a norm (both different); one was weighted and the other was unweighted. We generalized the dual forward-backward splitting algorithm and the primal-dual forward-backward splitting algorithm to solve the composite convex optimization problem (\ref{composite-sum}) and obtained the corresponding iterative algorithms under the defined inner product space. Furthermore, under certain conditions, we proved the convergence of the proposed iterative algorithms. We observed that if the weight vectors of the weighted inner product space are equal, then the corresponding iterative algorithm is equivalent to the iterative algorithm obtained in the unweighted inner product space. Furthermore, we applied the proposed iterative algorithms to the regularized general PICCS model (\ref{regu-GPICCS}) for CT image reconstruction from undersampled projection measurements. The numerical results indicated that the proposed iterative algorithms facilitate accurate reconstruction of the CT images. Compared with state-of-the-art ADMM, SPDP, and Pre-SPDP, the proposed iterative algorithms were shown to be more robust than those of others in terms of signal-to-noise ratio (SNR) versus number of iterations.

\section*{Competing interests}
The authors declare that they have no competing interests.

\section*{ACKNOWLEDGMENTS}
This work was supported by the National Natural Science Foundations
of China (11401293, 11661056), the Natural Science
Foundations of Jiangxi Province (20151BAB211010), the China Postdoctoral Science Foundation (2015M571989)
and the Jiangxi Province Postdoctoral Science Foundation (2015KY51).

% References
%\bibliographystyle{unsrt}
%\bibliography{klreference-en,essayfirst-en} %
 %

\end{document}